\documentclass[a4paper,12pt]{amsart}

\usepackage{a4wide}
\usepackage{pgf,tikz}
\usepackage{amssymb}
\usepackage{enumerate}

\usepackage{amsmath}

\newif\ifdetails
\detailstrue
\newcommand{\DETAIL}[1]%
{\ifdetails\par\fbox{\begin{minipage}{0.9\linewidth}\textit{Detail:}
      #1\end{minipage}}\par\fi}
\newcommand{\TODO}[1]%
{\ifdetails\par\fbox{\begin{minipage}{0.9\linewidth}\textbf{TODO:}
      #1\end{minipage}}\par\fi}

\usepackage{makecell}
\usepackage{relsize}

\newtheorem{lemma}{Lemma}

\newtheorem{theorem}[lemma]{Theorem}
\newtheorem{corollary}[lemma]{Corollary}
\theoremstyle{remark}

\newtheorem{remark}{Remark}
\newtheorem{conjecture}{Conjecture}

\DeclareMathOperator{\N}{\eta}

\DeclareMathOperator{\diam}{diam}
\usepackage{caption}
\usepackage{subcaption}
\usepackage{float} % provides H as float placement specifier
\usepackage[pdftex,a4paper,
citecolor = blue, colorlinks=true,urlcolor=blue]{hyperref}
\usepackage{mathrsfs}
\usetikzlibrary{arrows}

\newcommand{\old}[1]{{}}
\title[Nordhaus-Gaddum inequalities for the number of connected induced \ldots
]{Nordhaus-Gaddum inequalities for the number of connected induced subgraphs in graphs}

\author{Eric O. D. Andriantiana}
\author{Audace A. V. Dossou-Olory}
\thanks{This work is supported by grants from the National Research Foundation of South Africa: grant numbers 96310 and 118521}

\address{Eric O. D. Andriantiana \\ Department of Mathematics (Pure and Applied) \\ Rhodes University \\ PO Box 94, 6140, Grahamstown \\ South Africa}
\email{E.Andriantiana@ru.ac.za}

\address{Audace A. V. Dossou-Olory \\ Department of Mathematics and Applied Mathematics \\ University of Johannesburg \\ P.O. Box 524, Auckland Park, Johannesburg 2006\\ South Africa}
\email{audace@aims.ac.za}

\subjclass[2010]{Primary 05C30; secondary 05C05, 05C35}
\keywords{Nordhaus-Gaddum inequalities, induced subgraphs, connected graphs, trees, pendent vertices.}

\begin{document}

\begin{abstract}
Let $\eta(G)$ be the number of connected induced subgraphs in a graph $G$, and $\overline{G}$ the complement of $G$. We prove that $\eta(G)+\eta(\overline{G})$  is minimum, among all $n$-vertex graphs, if and only if $G$ has no induced path on four vertices. Since the $n$-vertex star $S_n$ with maximum degree $n-1$ is the unique tree of diameter $2$, $\eta(S_n)+\eta(\overline{S_n})$ is minimum among all $n$-vertex trees, while the maximum is shown to be achieved only by the tree whose degree sequence is $(\lceil n/2\rceil,\lfloor n/2\rfloor,1,\dots,1)$. Furthermore, we prove that every graph $G$ of order $n\geq 5$ and with maximum $\eta(G)+\eta(\overline{G})$ must have diameter at most $3$, no cut vertex and the property that $\overline{G}$ is also connected. In both cases of trees and graphs that have the same order, we find that if $\eta(G)$ is maximum then $\eta(G)+\eta(\overline{G})$ is minimum.

As corollaries to our results, we characterise the unique connected graph $G$ of given order and number of vertices of degree $1$, and the unique unicyclic (connected and has only one cycle) graphs %connected graph $G$ of a given order satisfying $|V(G)|=|E(G)|$ 
$G$ of a given order that minimises $\eta(G)+\eta(\overline{G})$.
\end{abstract}

\maketitle

\section{Introduction}\label{Sec:INTRO}

Graphs in this paper are simple, finite, and undirected. The vertex and edge sets of a graph $G$ are denoted by $V(G)$ and $E(G)$, respectively. The number of elements of a finite set $S$ is denoted by $|S|$. The number of vertices $|V(G)|$ will be referred to as the order of $G$.  An edge with ends $u$ and $v$ will be denoted by $uv$. A graph $H$ such that $V(H) \subseteq V(G)$ and $E(H) \subseteq E(G)$ is called a subgraph of $G$. We say that $H$ is \emph{induced} in $G$ if for every two vertices $u,v \in V(H)$, we have $uv \in E(H)$ if and only if $uv \in E(G)$. A graph $G$ is said to be connected if for every $u,v \in V(G)$, there is a sequence $v_0=u,v_1,\dots,v_t=v$ of vertices of $G$ such that $v_iv_{i+1} \in E(G)$ for all $0\leq i<t$. By $\overline{G}$, we mean the complement of $G$: the graphs $G$ and $\overline{G}$ have the same vertex set, and $uv \in E(G)$ if and only if $uv \notin E(\overline{G})$. A vertex of degree $1$ in a graph $G$ will be called a pendent vertex of $G$.

\medskip
Inequalities that relate the sum or product of a graph invariant to the same invariant of its complement are usually referred to as \emph{Nordhaus-Gaddum} type results, in honour to E.~A.~Nordhaus and J.~W.~Gaddum~\cite{nordhaus1956complementary} who first investigated such inequalities for the chromatic number $\chi(G)$ of $n$-vertex graphs $G$ in 1956.

\begin{theorem}[\cite{nordhaus1956complementary}]\label{Theo:NORD}
For a graph $G$ of order $n$, 
\begin{align*}
n+1\geq \chi(G) + \chi(\overline{G}) \geq 2\sqrt{n}\,.
\end{align*}
%where $\chi(G)$ is the chromatic number of $G$.
\end{theorem}

Following Theorem~\ref{Theo:NORD}, several relations of a similar type have been proposed for other graph invariants such as the Randi\'c index, the Zagreb index and the Wiener index~\cite{Wu} and  domination number \cite{Jaeger}; see also the survey~\cite{Aouchiche}. Recent results of a Nordhaus-Gaddum type include average of the maximum distance from a vertex (average eccentricity)~\cite{DANK2004}, vertex/edge connectivity~\cite{hellwig2008connectiv}, number of independent sets~\cite{Wei}, generalised edge-connectivity~\cite{Li}, Wiener polarity index~\cite{Hu}, among others.

\medskip
The aim of this paper is to present the corresponding Nordhaus-Gaddum type inequalities for the number of connected induced subgraphs in a graph. Let $\N(G)$ be the number of nonempty connected induced subgraphs in a graph $G$. If $G$ is acyclic, then $\N(G)$ counts precisely the \emph{number of subtrees} in $G$. Subtrees are fairly well studied in several contexts; see for example~\cite{GreedyWagner,AudacePeter,szekely2007binary,szekely2005subtrees}  for results on the number of subtrees,~\cite{graham1981trees} for trees that contain all small trees, \cite{jamison1983average,jamison1984monotonicity} for works on the orders of subtrees. Various estimates of the number of subtrees were reported in ~\cite{SBT1, SBT2}, and \cite{SBT3} provides a characterisation of all $n$-vertex trees with number of subtrees at least $5+n+2^{n-3}$. On the other hand, the extremal problems that estimate the number of connected subgraphs or connected induced subgraphs, or characterise graphs that have the maximum or minimum value in a prescribed class of graphs is only starting to attract attention of researchers~\cite{Alokshiya2019,CIS1,Maxwell2014}. Recently, Pandey and Patra~\cite{pandey2018extremizing} solved the cases of $n$-vertex connected graphs and unicyclic graphs, among other things. One of the authors of this paper started a systematic investigation of such extremal problems with respect to the number of connected induced subgraphs~\cite{Audacegenral2018,AudaceGirth2018,AudaceCut2019,AudaceUNICut2020}. In this paper, we study Nordhaus-Gaddum type inequalities for the number of connected induced subgraphs in a graph or tree, given the order.

\medskip
The rest of the paper is organised as follows: Section~\ref{Sec:Pre} aims to clarify notation and to prove some technical lemmas. We prove in  Section~\ref{Sec:Gen} that $\eta(G)+\eta(\overline{G})\geq 2^n+n-1$ for any graph $G$ with $n$ vertices, and that the bound is reached if and only if $G$ and $\overline{G}$ have no induced path (a tree with maximum degree at most $2$) on four vertices. Note that the $n$-vertex star (the tree with maximum degree $n-1$) satisfies this condition. It is also showed in the same section that if $n\geq 5$, then for $\eta(G)+\eta(\overline{G})$ to be maximum, $G$ and $\overline{G}$ must have diameter at most $3$ and no cut vertex.  

Since the star already reaches the minimum among all graphs of the same order, we only have to determine the $n$-vertex trees $T$ maximising $\eta(T)+\eta(\overline{T})$; see Section~\ref{Sec:Tree}. For $n>5$, we prove that the maximising tree is unique: it is the tree that has only two vertices of degree greater than $1$, with degree difference at most $1$.

\section{Preliminary results}\label{Sec:Pre}

In this section, we introduce more notation and prove some technical lemmas that are central to the proof of our main theorems. 

Let $G$ be a graph. We denote by $N_G(v)=\{u\in V(G): vu\in E(G)\}$ the set of all neighbors of $v$ in $G$, and we set $N_G[v]=N_G(v) \cup \{v\}$. We denote by $\deg_G (v)$ the degree $|N_G(v)|$ of a vertex $v$ in $G$. A path is a tree whose maximum degree is less than $3$. The number of edges in a path is also called its length. If $u,v\in V(G)$, then the length of a shortest path joining $u$ to $v$ is denoted by $d_G(u,v)$ (or $d(u,v)$ when there is no ambiguity) and called the distance between $u$ and $v$ of $G$. The maximum possible distance in a graph $G$ is called its diameter and denoted by $\diam(G)$.

If $S\subseteq V(G)$, then we write $\langle S \rangle_G$ for the induced subgraph in $G$ whose vertex set is $S$. By $G-S$ we mean the graph obtained from $G$ by deleting all vertices in $S$ (and all edges incident with them). For simplicity, we write $G-v$ instead of $G-\{v\}$. If $R\subseteq E(G)$, then we define $G-R$ as the graph obtained from $G$ by deleting all edges in $R$; we simply write $G-uv$ instead of $G-\{uv\}$. The set of all (nonempty) connected induced subgraphs in $G$ is denoted by $\mathcal{N}(G)$. We set $\N(G)=|\mathcal{N}(G)|$,
$
\N_k(G)=|\{H \in \mathcal{N}(G): |V(H)|=k\}|
$
and $$\N(G)_{v_1,\dots,v_{\ell}}=|\{H \in \mathcal{N}(G): v_i\in V(H)~\text{for all }1\leq i\leq \ell\}|\,.$$

We write $G_1 \cup G_2$ to mean the disjoint union of the graphs $G_1$ and $G_2$.  By $S_n$ we mean the star of order $n$, consisting of a vertex (called center) and $n-1$ pendent vertices attached to it. By $K_n$ we mean the complete graph of order $n$: it has the property that every two distinct vertices are adjacent. 

\medskip
The following well known remark will often be needed in this paper.

\begin{remark}\label{Rem:conn}
If $G$ is a disconnected graph, then $\overline{G}$ is connected and of diameter at most $2$.
\end{remark}

For a proof of Remark~\ref{Rem:conn}, note that every two distinct vertices that belong to different (connected) components of $G$ are adjacent in $\overline{G}$. Whenever $u,v \in V(G)$ belong to the same component of $G$ and  $uv \in E(G)$, we can take a vertex $w$ in a component of $G$ that does not contain $u$ (and $v$) to obtain the path $\langle \{u,w,v\}\rangle_{\overline{G}}$ between $u$ and $v$ in $\overline{G}$.

\medskip
Connected graphs whose complement is also connected will play an important role in this study. We define the set $\mathcal{N}(G,\overline{G})$ as
$$
\mathcal{N}(G,\overline{G}) =\{S \subseteq V(G): |S|>1,~\langle S \rangle_G ~\text{and}~ \langle S \rangle_{\overline{G}}~\text{are both connected} \}.$$ Let $G$ and $H$ be two graphs of the same order. Note that $$\N(G)+\N(\overline{G})=2^{|V(G)|}-1 +|V(G)| + |\mathcal{N}(G,\overline{G})|\,.$$
The term $2^{|V(G)|}-1$ counts all nonempty subsets of $|V(G)|$, the additional terms $|V(G)| + |\mathcal{N}(G,\overline{G})|$ account for all subsets of cardinality one and all elements of $\mathcal{N}(G,\overline{G})$, since they all induce connected subgraphs in $G$ and in $\overline{G}$. Thus, we have
$$\N(G)+\N(\overline{G})>\N(H)+\N(\overline{H})\text{ if and only if }|\mathcal{N}(G,\overline{G})|>|\mathcal{N}(H,\overline{H})|\,.$$

\medskip
The next two lemmas are also direct consequences of Remark~\ref{Rem:conn}.

\begin{lemma}\label{AuxLem}
Let $u,v$ be two distinct vertices of a graph $G$. We have $$ \N(G)_{u,v} +\N(\overline{G})_{u,v}  \geq 2^{n-2}\,.$$
The bound is attained if $u$ and $v$ belong to different components of $G$ or $\overline{G}$.
\end{lemma}

\begin{proof}
By Remark~\ref{Rem:conn}, $\N(G)_{u,v} + \N(\overline{G})_{u,v}  \geq |\{S\subseteq V(G):u,v\in S\}|=2^{n-2}.$ If $u$ and $v$ belong to different components of $G$ or $\overline{G}$, then exactly one of $\langle S\rangle _G$ and $\langle S\rangle_{\overline{G}}$ is disconnected. In this case, we get an equality.
\end{proof}

The next lemma shows, in particular, that any isolated vertex in a graph $G$ of order $n$ has the minimum contribution to $\N(G)+\N(\overline{G})$.

\begin{lemma}\label{Lem:Indiv}
Let $v$ be a vertex of a graph $G$ of order $n$. We have
$$\N(G)_v+\N(\overline{G})_v \geq 2^{n-1}+1\,.$$ Equality holds if $v$ is an isolated vertex of $G$ or $\overline{G}$. For $n \neq 1$, the inequality is strict provided that $G$ and $\overline{G}$ are both connected.
\end{lemma}

\begin{proof}
By Remark~\ref{Rem:conn},
$
  \N(G)_{v} + \N(\overline{G})_{v}  \geq |\{S\subseteq V(G):v\in S\}| +1 =2^{n-1}+1$,
since $\langle \{v\}\rangle_G=\langle \{v\}\rangle_{\overline{G}}$ counts twice. 
Assume that $n\neq 1$ and that $G$ and $\overline{G}$ are both connected. Then both $\{v\}$ and $V(G)$ induce connected subgraphs in $G$ and $\overline{G}$, hence
$\N(G)_{v} + \N(\overline{G})_{v}  \geq |\{S\subseteq V(G):v\in S\}| +2$. Thus, the strict inequality in the lemma.
\end{proof}

Note that $\N(G)_v+\N(\overline{G})_v$ also counts the number of connected induced subgraphs that one will loose in $G$ and $\overline{G}$ when deleting the vertex $v$. One can also ask a similar question when transfering an edge from $G$ to $\overline{G}$, that is when transforming $G$ and $\overline{G}$ to be $G-uv$ and $\overline{G}+uv$, respectively, for some $u,v\in V(G)$ and $uv\in E(G)$ : Remark~\ref{Rem:SameNeighb} sheds some light on this.

\begin{remark}\label{Rem:SameNeighb}
Let $uv \in E(G)$ such that the set of neighbors of $u$ in $G-uv$ and the set of neighbors of $v$ in $G-uv$ coincide. Then $\langle \{u,v\}\rangle_G$ is the only connected induced subgraph of $G$ that becomes disconnected when deleting the edge $uv$. Thus $\N(G)-\N(G-uv)=1$. On the other hand, the set of neighbors of $u$ and the set of neighbors of $v$ in $\overline{G}$ also coincide. 
Then $\langle \{u,v\}\rangle_{\overline{G}}$ is the only disconnected induced subgraph of $\overline{G}$ that becomes connected when adding the edge $uv$. Thus, $\N(\overline{G}+uv)-\N(\overline{G})=1$. It follows that
$$\N(G)+\N(\overline{G})= \N(G-uv)+\N(\overline{G}+uv)\,.$$
\end{remark}

\medskip
Combined with Lemma~\ref{Lem:Indiv}, the next lemma implies that disconnected graphs are not good candidate if $\N(G)+\N(\overline{G})$ is to be maximised among all $n$-vertex graphs $G$; see the proof of Theorem~\ref{Sec:Gen} in Section~\ref{Th:GandGbar}.

\begin{lemma}\label{Lem:Combine}
Let $G$ and $H$ be two disjoint connected graphs, and consider $u\in V(G)$ and $v\in V(H)$.  Let $B$ be the graph obtained from $G$ and $H$ by merging $u$ and $v$. 
\begin{itemize}
 \item[1)] $\N(B\cup K_1)\geq \N(G\cup H)$ with equality if and only if at least $G$ or $H$ is of order $1$. 
 \medskip
 \item[2)] $\N(B\cup K_1) + \N(\overline{B\cup K_1})\geq \N(G \cup H)+\N(\overline{G\cup H})
$ with equality if and only if $\deg_G (u)=|V(G)|-1$ and $\deg_H (v)=|V(H)|-1$.
\end{itemize}
\end{lemma}

\begin{proof}
For the inequality in 1), we have
\begin{align}\label{Eq:guh}
\begin{split}
\N(B\cup K_1)&=\N(G)+\N(H)+(\N(G)_u -1)(\N(H)_v-1)\\
& \geq \N(G)+\N(H) =\N(G\cup H)\,.
\end{split}
\end{align}
Equality holds if and only if $\N(G)_u =1$ or $\N(H)_v=1$.

It is left to prove 2). The graph $\overline{B\cup K_1}$ can be obtained by joining a new vertex $w$ to all vertices of $\overline{B}$, while $\overline{B}$ is the graph obtained from $\overline{G}$ and $\overline{H}$ by first merging vertices $u$ and $v$, and then adding all possible edges joining a vertex of $\overline{G}-u$ to a vertex of $\overline{H}-v$. Set $g=|V(G)|$ and $h=|V(H)|$. Note that $\N(\overline{B\cup K_1})=2^{g+h-1} + \N(\overline{B})$. Let us evaluate $\N(\overline{B})$. The connected induced subgraphs of $\overline{B}$ that cross $u$ (and $v$), i.e. that contain a vertex of $G-u$ and a vertex of $H-v$, can be categorised as follows:
\begin{itemize}
\item those that do not contain $v$ (and $u$): there are $(2^{g-1}-1) (2^{h-1}-1)$ in total;
\item those that contain $v$ and a nonempty set of neighbors of $v$ in $\overline{H}$: there are
$$ (2^{\deg_{\overline{H}} (v)}-1) (2^{g-1}-1) 2^{h-1-\deg_{\overline{H}} (v)}
$$ in total;
\item those that contain $v$ (and $u$) and no neighbor of $v$ in $\overline{H}$: there are
$$ (2^{\deg_{\overline{G}} (u)}-1) (2^{h-1-\deg_{\overline{H}} (v)}-1) 2^{g-1-\deg_{\overline{G}} (u)}
$$ in total.
\end{itemize}
On the other hand, there are precisely $\N(\overline{G})+\N(\overline{H})-1$ connected induced subgraphs of $\overline{B}$ that do not cross $u$ (and $v$). Therefore, 
\begin{align}\label{ForSix}
\N(\overline{B\cup K_1})&= 2^{g+h-1} + (2^{g-1}-1) (2^{h-1}-1) + (2^{\deg_{\overline{H}} (v)}-1) (2^{g-1}-1) 2^{h-1-\deg_{\overline{H}} (v)} \nonumber\\
& \hspace*{1.7cm} + (2^{\deg_{\overline{G}} (u)}-1) (2^{h-1-\deg_{\overline{H}} (v)}-1) 2^{g-1-\deg_{\overline{G}} (u)}+\N(\overline{G})+\N(\overline{H})-1 \nonumber\\
&=\N(\overline{G})+\N(\overline{H}) + (2^g-1) (2^h-1)-(2^{g-1-\deg_{\overline{G}} (u)}-1)(2^{h-1-\deg_{\overline{H}} (v)}-1) \nonumber \\
&=\N(\overline{G\cup H}) -(2^{\deg_G (u)}-1)(2^{\deg_H (v)}-1)\,.
\end{align}
It follows from~\eqref{Eq:guh} and~\eqref{ForSix} that	
\begin{align*}
\N(B\cup K_1) + \N(\overline{B\cup K_1})&=\N(G \cup H)+\N(\overline{G\cup H})\\
& \hspace*{1cm} + (\N(G)_u -1)(\N(H)_v-1) -(2^{\deg_G (u)}-1)(2^{\deg_H (v)}-1)\\
& \geq \N(G \cup H)+\N(\overline{G\cup H})\,.
\end{align*}
The last inequality stems from the observation that a vertex and any subset of its neighborhood always induce a connected graph, hence $\N(G)_u\geq 2^{\deg_G(u)}$ and $\N(H)_v\geq 2^{\deg_H(v)}$.
Equality happens in the formula if and only if $\N(G)_u=2^{\deg_G (u)}$ and $\N(H)_v=2^{\deg_H (v)}$, that is only if $\deg_G (u)=g-1$ and $\deg_H (v)=h-1$.
\end{proof}

\medskip
We are now ready to state and prove our main theorems.

\section{$n$-vertex graphs}\label{Sec:Gen}

In this section, we give a full characterisation of those $n$-vertex graphs $G$ that satisfy
$$\N(H)+\N(\overline{H})\geq \N(G)+\N(\overline{G})$$ for all $n$-vertex graphs $H$. Furthermore, we manage to obtain some partial characterisations of an $n$-vertex graph $G$ that satisfies
$$\N(H)+\N(\overline{H}) \leq \N(G)+\N(\overline{G})$$ for all graphs $H$ of order $n$. In the former case, $G$ will be referred to as a \emph{minimal} graph while in the latter case, $G$ will be referred to as a \emph{maximal} graph. 

Let us bear in mind that every shortest path in a graph $G$ is an induced path of $G$.

\begin{theorem}\label{Th:GandGbar}
For a graph $G$ of order $n$, we have
\begin{align}\label{Eq:Nk}
\N_k(G) +\N_k(\overline{G}) \geq \binom{n}{k} \text{ for all } k>1\,.
\end{align}
In particular,
\begin{align}\label{Eq:NOR}
\N(G) +\N(\overline{G}) \geq 2^n +n -1.
\end{align}
Equality holds in~\eqref{Eq:NOR} if and only if $G$ or $\overline{G}$ has no induced path of length at least $3$. This includes, for instance, all $n$-vertex disjoint union of complete graphs and their respective complements, which are complete multipartite graphs such as stars.
\end{theorem}

\begin{proof}
Inequality~\eqref{Eq:Nk} is a straightforward consequence of Remark~\ref{Rem:conn}. Summing up~\eqref{Eq:Nk} for all possible values of $k$, we obtain~\eqref{Eq:NOR} since the subgraph of order $1$ counts for both $G$ and $\overline{G}$. If $G$ has an induced path $P_4$ of length $3$, then the set $\mathcal{N}(G,\overline{G})$ is non-empty since $\overline{P_4}$ is isomorphic to $P_4$, and it is connected. Thus $\N(G) +\N(\overline{G}) > 2^n +n -1$. Conversely, suppose that $G$ does not have an induced path $P_4$. Then $\diam (G)\leq 2$. We are going to show that $\mathcal{N}(G,\overline{G})$ is empty, thus proving the theorem. Let $H$ be a connected induced subgraph of $G$ such that $|V(H)| > 1$. Then $H$ has no induced $P_4$ and thus $\diam (H) \leq 2$. If $\diam(H)=1$, then $H$ is a complete graph and thus $\overline{H}$ is disconnected. Otherwise, $\diam(H)=2$. Let $u,v\in V(H)$ such that $d_{H}(u,v)=2$. Set $A=N_H(u)\cap N_H(v)$. Clearly, $A$ is non-empty.

\medskip
The following four claims reveal some information on the structure of $H$, enough to see that there is a subset $B$ of $V(H)$ whose elements are adjacent to all vertices in $V(H)\smallsetminus B.$ This means that $\overline{H}$ is disconnected.

\medskip

\textbf{Claim 1}: \emph{For every $x\in N_H(u)\smallsetminus A$ and $y\in A$, $xy$ is an edge of $H$. Likewise, for every $x \in N_H(v)\smallsetminus A$ and $y \in A$, $xy$ is an edge of $H$.}
\begin{proof}[Proof of Claim 1]
Simply note that if $xy \notin E(H)$, then vertices $x,u,y,v$ form an induced path $P_4$, which is a contradiction. 
\end{proof}

\medskip
\textbf{Claim 2}: \emph{Let $z\in V(H)\smallsetminus (N_H[v]\cup N_H[u])$. Then $z$ has a neighbor in $A$.}
\begin{proof}[Proof of Claim 2]
Suppose to the contrary that $z$ has no neighbor in $A$. By $\diam(H)=2$, let $P=\langle \{ z,x_1,y\}\rangle_H$ be a shortest path from $z$ to a vertex $y \in A$. Note that $x_1 \notin \{u,v\}\cup A$. If $x_1 \notin N_H(u)$ then $\langle \{z,x_1,y,u\} \rangle_H =P_4$. Similarly, if $x_1 \notin N_H(v)$ then $\langle \{z,x_1,y,v\} \rangle_H =P_4$. Thus, we get a contradiction in both situations.
\end{proof}

\medskip

\textbf{Claim 3}: \emph{By Claim~2, let $w$ be a neighbor of $z\in V(H)\smallsetminus (N_H[u] \cup N_H[v])$ in $A$. Then $w$ is adjacent to all vertices in $A$ that are not a neighbor of $z$ in $H$.}
\begin{proof}[Proof of Claim 3]
Suppose that there is a $y \in A \smallsetminus N_H(z)$ such that $w \notin N_H(y)$. Then $\langle \{z,w,v,y\} \rangle_H=P_4$, a contradiction.
\end{proof}

\medskip

\textbf{Claim 4}: \emph{Let $z,z'\in V(H)\smallsetminus (N_H[u]\cup N_H[v])$ such that $z\neq z'$. Then either $A \cap N_H(z) \subseteq A \cap N_H(z')$ or $A \cap N_H(z') \subseteq A \cap N_H(z)$.}
\begin{proof}[Proof of Claim 4]
By Claim~2, $A \cap N_H(z) \neq \emptyset$ and $A \cap N_H(z') \neq \emptyset$. Suppose that Claim~4 is not true. Then there is $w' \in A \cap N_H(z')$ such that $w' \notin A \cap N_H(z)$, and there is $w'' \in A \cap N_H(z)$ such that $w'' \notin A \cap N_H(z')$. By Claim~3, $w'$ is adjacent to $w''$. If $z$ and $z'$ are not adjacent, then $\langle \{z,w'',w',z'\} \rangle=P_4$ (a contradiction). If $z$ and $z'$ are adjacent, then $\langle \{z',z,w'',v\} \rangle =P_4$ (a contradiction). This proves the claim.
\end{proof}

\medskip
Set $V(H)\smallsetminus (N_H[u]\cup N_H[v])=\{z_1,\ldots,z_k\}$. We get from Claims~2 and~4 that
\begin{align}
\label{Eq:Subs}
\emptyset\neq A \cap N_H(z_{\sigma(1)})\subseteq  \cdots \subseteq A \cap N_H(z_{\sigma(k)})\,
\end{align}
for some permutation $\sigma$ of the set $\{1,\dots,k\}$.
Claims~1 through~3, and the relation \eqref{Eq:Subs} imply that each $q \in A\cap N_H(z_1)$ is adjacent to all vertices in $V(H)\smallsetminus (A\cap N_H(z_{\sigma(1)}))$. Hence, there is no edge of $\overline{H}$ from a vertex in $B:=A\cap N_H(z_{\sigma(1)})$ to a vertex in $V(H) \smallsetminus B$. That is $\overline{H}$ is disconnected; thus the set $\mathcal{N}(G,\overline{G})$ is empty, which was to be proved.
\end{proof}

\medskip
We remark that for $n>1$, equality never holds in~\eqref{Eq:NOR} if $G$ and $\overline{G}$ are both connected. 

Let $\mathbb{G}_n$ be the set of all graphs of order $n$ and possibly other specific properties. If there is $G\in\mathbb{G}_n$ such that $G$ has no induced path of length $3$, then by Theorem~\ref{Th:GandGbar}, 
$$
\eta(G)+\eta(\overline{G})
=\min\left\lbrace \eta(H)+\eta(\overline{H}): H\in \mathbb{G}_n \right\rbrace \,.
$$
Hence the following corollaries are immediate consequences of Theorem~\ref{Th:GandGbar}.

\begin{corollary}
Let $R_{n,k}$ be the graph of order $n>4$ obtained by merging a vertex of the complete graph $K_{n-k}$ with the center of the star $S_{k+1}$, for some $k<n-2$.  For all connected graphs $G\neq R_{n,k}$ with $n$ vertices of which $k$ are pendent, we have
$$
\eta(R_{n,k})+\eta(\overline{R_{n,k}})
<  \eta(G)+\eta(\overline{G}).
$$
\end{corollary}

It is not hard to show that $\eta(R_{n,k})\geq \eta(G)$ holds for all graphs $G$ with order $n>4$ and $k<n-2$ pendent vertices: Starting from $G$, add all possible edges between any pair of vertices that are not pendent to obtain the subgraph $K_{n-k}$, then move all pendent vertices such that they are all attached to one vertex of $K_{n-k}$. Each of these transformations does not decrease the number of connected induced subgraphs.

\begin{corollary}\label{Coro:Treeslower}
We have
$$
\eta(S_n)+\eta(\overline{S_n})< \eta(T)+\eta(\overline{T})
$$
for all $n$-vertex trees $T$ that are different from $S_n$.
\end{corollary}

Note that $\eta(S_n)>\eta(T)$ holds for all $n$-vertex trees $T\neq S_n$; see \cite[Corollary 3]{GreedyWagner}.

\medskip
A graph $H$ is said to be unicyclic if it is connected and $|V(H)|=|E(H)|$.

\begin{corollary}
Let $U_{n}$ be the graph obtained from $K_3$ and $S_{n-2}$ by merging a vertex of $K_3$ with the center of $S_{n-2}$. Then 
$$
\eta(U_n)+\eta(\overline{U_n})< \eta(H)+\eta(\overline{H})
$$
for all $n>2$-vertex unicyclic graphs $H$ different from $U_n$.
\end{corollary}
It should be mentioned that $\eta(U_n)>\eta(G)$ holds for all unicyclic graphs $G\neq  U_n$ of order $n$; see \cite[Theorem 10]{Audacegenral2018}.

\medskip
We say that a path $P$ in a graph $G$ is \emph{pendent} if one of its ends has degree $1$ in $G$ and all internal vertices have degree $2$ in $G$. A cut edge of a graph $G$ is an edge whose deletion increases the number of connected components of $G$ at least by one. It is clear that $e\in E(G)$ is a cut edge of $G$ if and only if $e$ is not contained in a cycle of $G$.

A cut vertex of a graph $G$ is a vertex whose deletion increases the number of (connected) components of $G$ at least by one.

\medskip
In the next theorem, we manage to obtain some partial characterisations of an $n$-vertex graph $G$ that maximises $\N(G) + \N(\overline{G})$.
 
\begin{theorem}\label{GeneGraphs}
Let $G$ be a graph of order $n> 4$. If  $\N(H)+\N(\overline{H}) \leq \N(G)+\N(\overline{G})$ for all $n$-vertex graphs $H$, then the following must hold:
\begin{enumerate}[1)] 
\item $G$ and $\overline{G}$ are connected.
\item $G$ and $\overline{G}$ have diameter $2$ or $3$.
\item $G$ and $\overline{G}$ have no pendent vertex and no cut vertex.
\end{enumerate}
\end{theorem}

We will need the following lemma in the proof of Theorem~\ref{GeneGraphs}.

\begin{lemma}\label{Nopend2}
Let $G$ be a graph of order $n>4$ and with diameter $2$ or $3$ such that $\overline{G}$ is also connected. If $\N(H)+\N(\overline{H}) \leq \N(G)+\N(\overline{G})$ for all $n$-vertex graphs $H$, then neither $G$ nor $\overline{G}$ contains a pendent path of length $2$.
\end{lemma}

\begin{proof}
The case when $\diam(G)=2$ is trivial. For the rest of the proof, we assume that $\diam(G)=3$. Suppose that $G$ contains a pendent path of length $2$, say $P=\langle v_0,v_1,v_2\rangle_G$ attached at $v_0$, that is $\deg_G(v_0)\geq 3$ and $\deg_G(v_2)=1$. Let $G_1=G-v_1-v_2$ and $x\in V(G_1-v_0)$. Define $G'=G+xv_2$. Clearly, $G'$ and $\overline{G'}=\overline{G}-xv_2$ are connected. Since $\diam(G)=3$, we have $x\in N_{G_1}(v_0)=V(G_1-v_0)$. Let us prove that $ |\mathcal{N}(G,\overline{G})| < |\mathcal{N}(G',\overline{G'})|$, thereby contradicting the maximality of $G$.

Suppose that $S \in \mathcal{N}(G,\overline{G})$ and that $S \notin \mathcal{N}(G',\overline{G'})$. Then $xv_2$ is a cut edge of $H:=\langle S\rangle_{\overline{G}}$. Any path in $G$ joining $x$ to $v_2$ must contain $v_0$ and $v_1$. Therefore, we must have $x, v_0,v_1,v_2\in V(H)= V(\overline{H})$ for the subgraph $\overline{H}$ of $G$ to be connected.
Denote by $H_1$ and $H_2$ the two components of $H-xv_2$ with $x\in V(H_1)$ and $v_2\in V(H_2)$. Thus $v_1\in V(H_1)$ as $xv_1\in E(H)$. Moreover, $V(H_1)$ cannot contain a vertex $y$ other than $x$ and $v_1$, otherwise there is a path $R$ joining $x$ to $y$ in $H_1$ and so $R+yv_2$ is an alternative path joining $x$ to $v_2$ in $H$. Hence, $V(H_1)=\{x,v_1\}$. Similarly, $V(H_2-v_2-v_0)=\emptyset$. Otherwise, if $z\in V(H_2-v_2-v_0)$, then there is a path $P$ that joins $v_2$ to $z$ in $H_2$. Then $P+zv_1+v_1x$ is an alternative path that joins $x$ to $v_2$. So, we must have $V(H)=\{v_0,v_1,v_2,x\}$, and 
\begin{align}
\label{Eq:PP1}
|\mathcal{N}(G,\overline{G})\smallsetminus\mathcal{N}(G',\overline{G'})|=1\,. %\textcolor{blue} {2^{\deg_{G_1}(x)-1}=2^{\deg_G(x)-1}}\,.
\end{align}
On the other hand, for every $u  \in N_G(x)\smallsetminus\{v_0\}$, the graph $Q=\langle u,x,v_2,v_1 \rangle_{G'}$ is a path and thus $\{u,x,v_2,v_1\} \in\mathcal{N}(G',\overline{G'})$ while $\{u,x,v_2,v_1\} \notin \mathcal{N}(G,\overline{G})$. There are $2^{\deg_G(x)-1}-1$ choices for $Q$. Moreover, for every $y\in V(G_1-x-v_0)$, it is easy to check that 
		\begin{itemize}
			\item $\{y,x,v_0,v_2\}\in \mathcal{N}(G',\overline{G'})$ and $\{y,x,v_0,v_2\} \notin \mathcal{N}(G,\overline{G})$ if $y\notin N_G(x)$,
			\item $\{y,x,v_1,v_2\}\in \mathcal{N}(G',\overline{G'})$ and $\{y,x,v_1,v_2\} \notin \mathcal{N}(G,\overline{G})$ if $y\in N_G(x)$.
		\end{itemize}
Thus 
\begin{align}
\label{Eq:PP2}
|\mathcal{N}(G',\overline{G'})\smallsetminus \mathcal{N}(G,\overline{G})|\geq 2^{\deg_G(x)-1}-1 + |V(G_1-x-v_0)|\,.
\end{align}
Furthermore,~\eqref{Eq:PP1} and~\eqref{Eq:PP2} imply that  $|\mathcal{N}(G',\overline{G'})|>|\mathcal{N}(G,\overline{G})|$ for $|V(G_1-x-v_0)|> 1$ (i.e. $|V(G)|>5$), and contradicts the maximality of $G$. For $|V(G)|=5$, Table~\ref{Table1} shows that the unique maximal graph is the cycle. This completes the proof of the lemma.
\end{proof}

\begin{table}[h!]
	\caption{The only maximal graphs of order $5,6$ and $7$.}\label{Table1}
	\begin{tabular}{|c|c|c|c|}
		\hline 
		& $n=5$ & $n=6$ & $n=7$\\
		\hline 
		$G$ & \begin{tikzpicture}
		\node[fill=black,circle,inner sep=1pt] (t1) at (1,0) {};
		\node[fill=black,circle,inner sep=1pt] (t2) at (0.3,0.95) {};         	
		\node[fill=black,circle,inner sep=1pt] (t3) at (-0.80,0.58) {};
		\node[fill=black,circle,inner sep=1pt] (t4) at (-0.80,-0.58) {};
		\node[fill=black,circle,inner sep=1pt] (t5) at (0.30,-0.95) {};
		\draw (t1)--(t2)--(t3)--(t4)--(t5)--(t1);
		\end{tikzpicture}
		&
		\begin{tikzpicture}
		\node[fill=black,circle,inner sep=1pt] (t1) at (1,0) {};
		\node[fill=black,circle,inner sep=1pt] (t2) at (0.3,0.95) {};         	
		\node[fill=black,circle,inner sep=1pt] (t3) at (-0.80,0.58) {};
		\node[fill=black,circle,inner sep=1pt] (t4) at (-0.80,-0.58) {};
		\node[fill=black,circle,inner sep=1pt] (t5) at (0.30,-0.95) {};
		\node[fill=black,circle,inner sep=1pt] (t6) at (0,0) {};
		\draw (t1)--(t2)--(t3)--(t4)--(t5)--(t1);
		\draw (t2)--(t6)--(t5);
		\draw (t6)--(t1);
		\end{tikzpicture}
		&
		\begin{tikzpicture}
		\node[fill=black,circle,inner sep=1pt] (t1) at (1,0) {};
		\node[fill=black,circle,inner sep=1pt] (t2) at (0.3,0.95) {};         	
		\node[fill=black,circle,inner sep=1pt] (t3) at (-0.80,0.58) {};
		\node[fill=black,circle,inner sep=1pt] (t4) at (-0.80,-0.58) {};
		\node[fill=black,circle,inner sep=1pt] (t5) at (0.30,-0.95) {};
		\node[fill=black,circle,inner sep=1pt] (t6) at (0.5,0.3) {};
		\node[fill=black,circle,inner sep=1pt] (t7) at (-0.2,0) {};
		\draw (t1)--(t2)--(t3)--(t4)--(t5)--(t1);
		\draw (t2)--(t6)--(t5);
		\draw (t7)--(t6)--(t1);
		\draw (t3)--(t7)--(t4);
		\end{tikzpicture} \\
		\hline
		$\overline{G}$
		& 
		\begin{tikzpicture}
		\node[fill=black,circle,inner sep=1pt] (t1) at (1,0) {};
		\node[fill=black,circle,inner sep=1pt] (t2) at (0.3,0.95) {};         	
		\node[fill=black,circle,inner sep=1pt] (t3) at (-0.80,0.58) {};
		\node[fill=black,circle,inner sep=1pt] (t4) at (-0.80,-0.58) {};
		\node[fill=black,circle,inner sep=1pt] (t5) at (0.30,-0.95) {};
		\draw (t1)--(t3)--(t5)--(t2)--(t4)--(t1);
		\end{tikzpicture}
		& 
		\begin{tikzpicture}
		\node[fill=black,circle,inner sep=1pt] (t1) at (1,0) {};
		\node[fill=black,circle,inner sep=1pt] (t2) at (0.3,0.95) {};         	
		\node[fill=black,circle,inner sep=1pt] (t3) at (-0.80,0.58) {};
		\node[fill=black,circle,inner sep=1pt] (t4) at (-0.80,-0.58) {};
		\node[fill=black,circle,inner sep=1pt] (t5) at (0.30,-0.95) {};
		\node[fill=black,circle,inner sep=1pt] (t6) at (0,0) {};
		\draw (t1)--(t3)--(t5)--(t2)--(t4)--(t1);
		\draw (t3)--(t6)--(t4);
		\end{tikzpicture}
		&
		\begin{tikzpicture}
		\node[fill=black,circle,inner sep=1pt] (t1) at (1,0) {};
		\node[fill=black,circle,inner sep=1pt] (t2) at (0.3,0.95) {};         	
		\node[fill=black,circle,inner sep=1pt] (t3) at (-0.80,0.58) {};
		\node[fill=black,circle,inner sep=1pt] (t4) at (-0.80,-0.58) {};
		\node[fill=black,circle,inner sep=1pt] (t5) at (0.30,-0.95) {};
		\node[fill=black,circle,inner sep=1pt] (t6) at (0.5,0.3) {};
		\node[fill=black,circle,inner sep=1pt] (t7) at (-0.2,0) {};
		\draw (t1)--(t3)--(t5)--(t2)--(t4)--(t1);
		\draw (t3)--(t6)--(t4);
		\draw (t1)--(t7)--(t2);
		\draw (t5)--(t7);
		\end{tikzpicture}\\
		\hline
	\end{tabular}
\end{table}

\medskip

\begin{proof}[Proof of Theorem~\ref{GeneGraphs}]
From a computer search among all graphs of order $5,6,7$, respectively, we noticed that the three properties in Theorem~\ref{GeneGraphs} are consistently satisfied by the maximal graphs (see Table \ref{Table1}). 

For the proof, we assume that $n \geq 8$. Let $G$ be a $n$-vertex maximal graph, i.e., that maximises $\N(G)+\N(\overline{G})$ among all graphs of order $n$. By Remark~\ref{Rem:conn}, at least one of $G$ and $\overline{G}$ is connected.

We begin with the proof of 1). Suppose that $\overline{G}$ is connected and that $G$ is disconnected. For the sake of induction on the number of components of $G$, let us assume that $G$ has exactly two components, say $G'$ and $G''$. Let the graph $G_1$ be constructed from $G=G' \cup G''$ by first merging a fixed $u\in V( G')$ and a fixed $v\in V(G'')$, and adding a new isolated vertex $w$. Of course, $G_1$ is disconnected and $\overline{G_1}$ is connected. By Lemma~\ref{Lem:Combine}, we have
$$\N(G_1)+ \N(\overline{G_1})\geq \N(G)+ \N(\overline{G})\,.$$ If $\deg_{\overline{G_1}}(s)=1$ for all $s\in V(\overline{G_1}-w)$, then $G_1=K_{|V(G)|-1}\cup K_1$, and thus $\N(G_1)+\N(\overline{G_1})$ is minimum among all $n$-vertex graphs since $G_1$ is an union of complete graphs; see Theorem~\ref{Th:GandGbar}. Therefore, there exists a vertex $s_0\neq w$ such that $\deg_{\overline{G_1}}(s_0)\geq 2$. Define $G_2$ to be the graph $G_1+ws_0$. Then $G_2$ and $\overline{G_2}=\overline{G_1}-ws_0$ are both connected. Since $w$ is isolated in $G_1$ and $G_2-w=G_1-w$ (isomorphic graphs), we deduce from Lemma~\ref{Lem:Indiv} that $\N(G_2)+\N(\overline{G_2})> \N(G_1)+\N(G_1)$. This also implies that $$\N(G_2)+\N(\overline{G_2})> \N(G)+\N(G),$$ contradicting the maximality of $G$. The induction step follows from the fact that for any two graphs $H$ and $K$, 
$$
\eta(H\cup K)+\eta(\overline{H\cup K})
=\big(\eta(H)+\eta(\overline{H})\big) + \big(\eta(K)+\eta(\overline{K})\big)+ (2^{|V(H)|}-1)(2^{|V(K)|}-1)\,.
$$
This ends the proof of 1). For the rest of the proof, we assume that $G$ and $\overline{G}$ are both connected. 

\medskip
Let us prove 2). For contradiction, suppose that one of the graphs $G$ and $\overline{G}$, say $G$ has two vertices $u$ and $v$ such that $d_G(u,v)>3$. Then $uv \notin E(G)$. Consider a new graph $G_1:=G+uv$ and note that $\overline{G_1}=\overline{G}-uv$ is still connected. This is because $uv$ is contained in a cycle of $\overline{G}$: since $d_G(u,v)\geq 4$, the complement $\overline{P}$ of a shortest path $P$ from $u$ to $v$ in $G$ contains a cycle that contains the edge $uv$. Now we show that $\mathcal{N}(G,\overline{G})$ is a strict subset of $\mathcal{N}(G_1,\overline{G_1})$.

Suppose that $S \in \mathcal{N}(G,\overline{G})$ and that $S \notin \mathcal{N}(G_1,\overline{G_1})$. Then $u$ and $v$ necessarily belong to $S$ and $uv$ is a cut edge of $\langle S\rangle_{\overline{G}}$. However, since $\langle S\rangle_{G}$ is connected, it contains a shortest path $P$ of length greater than $3$ that joins $u$ to $v$. Denote by $u,u_1,u_2,u_3$ the first four vertices of $P$ from $u$. Then $uu_3,u_3u_1,u_1v\in E(\langle S \rangle_{\overline{G}})$, which contradicts the fact that $uv$ is a cut edge of $\langle S \rangle_{\overline{G}}$. Hence, the inclusion $ \mathcal{N}(G,\overline{G})\subseteq \mathcal{N}(G_1,\overline{G_1})$ is obtained. It remains to show that this inclusion is strict.

Then the set $\{v,u,u_1,u_2\}$ induces a path in $G_1$ and a disconnected graph in $G$. Since $\langle\{v,u,u_1,u_2\}\rangle_{\overline{G}}$ is also a path, we have $\{v,u,u_1,u_2\} \in \mathcal{N}(G_1,\overline{G_1})$, while $\{v,u,u_1,u_2\}\notin \mathcal{N}(G,\overline{G})$. Hence, the strict inclusion $ \mathcal{N}(G,\overline{G})\subset \mathcal{N}(G_1,\overline{G_1})$ follows. This proves that $|\mathcal{N}(G_1,\overline{G_1})|> |\mathcal{N}(G,\overline{G})|$, contradicting the maximality of $G$. This completes the proof of 2). For the rest of the proof, we assume that 1) and 2) hold.

\medskip
The proof of 3) is done in two main steps. In the first step, we prove that every cut vertex of $G$ (resp. $\overline{G}$) is adjacent to some pendent vertex of $G$ (resp. $\overline{G}$). In the second step, we prove that neither $G$ nor $\overline{G}$ has a pendent vertex.

\medskip
Without loss of generality, assume that $c$ is a cut vertex of $G$. Consider a component $U_1$ of $G-c$ chosen such that all vertices of $U_1$ are neighbors of $c$, that is $U_1\subseteq N_G(c)$. Such a choice is possible, given that $\diam(G)\leq 3$. Denote by $U_2$ the graph $(G-c)-V(U_1)$, and by $U_2^+$ the graph $G-V(U_1)$. Then there is no edge between a vertex of $U_1$ and a vertex of $U_2$ in $G$. Moreover, $U_2$ has a vertex $u_2$ that is not adjacent to $c$ in $G$, otherwise $c$ is isolated in $\overline{G}$: a contradiction to the assumption that $\overline{G}$ is connected. We fix $u_2$. Let $u_1\in V(U_1)$ be chosen to have maximum degree in $G$ among all vertices of $U_1$. Let us show that $u_1$ is a pendent vertex of $G$.

Suppose not. We distinguish two main cases.\\
\textbf{Case 1}: $\deg_G(u_1)=2$ and $u_2$ is the only vertex of $U_2$ that is not adjacent to $c$ in $G$.\\
Then $|U_1|=2$. Define a new graph $H:=G-u_1c$. Then the number of connected induced subgraphs of $G-u_2$ that have $u_1c$ as a cut edge is $2^{\deg_{U_2^+}(c)}=2^{\deg_{G}(c) -2}$, they are induced by $W\cup\{u_1,c\}$ for some $W\subseteq U_2$. The number of $u_2$-containing connected induced subgraphs of $G$ that have $u_1c$ as a cut edge is $$(2^{\deg_{U_2}(u_2)} -1)2^{\deg_{U_2^+}(c) - \deg_{U_2}(u_2)}=(2^{\deg_{G}(u_2)} -1)2^{\deg_{G}(c) -2 - \deg_{G}(u_2)}\,.$$ The sum of these two quantities is $A_1:=2^{\deg_{G}(c) -1} -2^{\deg_{G}(c) -2 - \deg_{G}(u_2)}$. On the other hand, the number of connected induced subgraphs of $\overline{H}=\overline{G}+u_1c$ that have $u_1c$ as a cut edge is $A_2:=2^{\deg_{U_2}(c)+1} -1=2^{\deg_{G}(c) -1} -1$ since none of these subgraphs can contain $u_2$. We deduce that
\begin{align*}
(\eta(G)+\eta(\overline{G})) - (\eta(H)+\eta(\overline{H}))=A_1-A_2=-2^{\deg_{G}(c) -2 - \deg_{G}(u_2)} +1 \leq 0\,.
\end{align*}
Even in the case of equality, we get a contradiction to 2) since $d_H(u_1,u_2)=4$.\\
\textbf{Case 2}: $\deg_G(u_1)>2$ or $u_2$ is not the only vertex of $U_2$ that is not adjacent to $c$ in $G$.\\
Define a new graph $H:=G+u_1u_2$. Let us bound $|\mathcal{N}(G,\overline{G})\smallsetminus\mathcal{N}(H,\overline{H})|$ from above. Let $S\in \mathcal{N}(G,\overline{G})$ such that $S\notin \mathcal{N}(H,\overline{H})$. Then the graphs $\langle S\rangle_{G}, \langle S\rangle_{H}$ and $K:=\langle S\rangle_{\overline{G}}$ are all connected, while $\langle S\rangle_{\overline{H}}$ is disconnected. Thus $u_1u_2$ is a cut edge of $K$ and $c \in V(K)$, otherwise $\langle S\rangle_{G}$ would be disconnected. We recall that in $\overline{G}$, there is an edge between every vertex of $U_1$ and every vertex of $U_2$. So we cannot have $u_1' \in V(U_1), u_2' \in V(U_2)$ for some $u_1'\neq u_1, u_2'\neq u_2$ such that $u_1',u_2'$ are both vertices of $K$, as otherwise $c$ would not be a cut vertex of $G$. Moreover, $u_2$ is not the only element of $U_2$ in $K$, otherwise $u_2$ is isolated in $\overline{K}$. Therefore, there is $u_2' \in V(U_2)$ such that $u_2' \neq u_2$ and $u_2'\in V(K)$, and thus $V(U_1) \cap V(K)=\{u_1\}$. Also, all vertices $x\in V(K)\cap V(U_2-u_2)$ are neighbors of $u_2$ in $G$, otherwise $\langle u_1,u_2,x\rangle_K$ is a triangle in $K$: a contradiction to the fact that $u_1u_2$ is a cut edge of $K$. Putting everything together, we deduce that the number of sets $S$ satisfying both $S\in \mathcal{N}(G,\overline{G})$ and $S \notin \mathcal{N}(H,\overline{H})$ is
$$
|\mathcal{N}(G,\overline{G})\smallsetminus\mathcal{N}(H,\overline{H})|\leq 2^{\deg_G{u_2}}-1.
$$ Let us bound $|\mathcal{N}(H,\overline{H})\smallsetminus\mathcal{N}(G,\overline{G})|$ from below. Since $\deg_G(u_1)=1+\deg_{U_1}(u_1)\geq 2$, let $z_1\in V(U_1)$ be a fixed neighbor of $u_1$ in $U_1$. For any nonempty subset $S' \subseteq N_G(u_2)$, we have the following:
\begin{itemize}
 \item $\langle \{z_1,u_1,u_2\}\cup S'\rangle_G$ is disconnected, because it contains $u_1$ and $u_2$ but not $c$. Thus $\{z_1,u_1,u_2\}\cup S'\notin \mathcal{N}(G,\overline{G}).$
 \item $\langle \{z_1,u_1,u_2\}\cup S'\rangle_H$ and $\langle \{z_1,u_1,u_2\}\cup S'\rangle_{\overline{H}}$ are connected since $u_2u_1,u_1z_1$ are edges in $H$ and $z_1u_2,z_1x,u_1x$ are edges in $\overline{H}$ for any $x\in S'$. So $\{z_1,u_1,u_2\}\cup S'\in \mathcal{N}(H,\overline{H}).$
\end{itemize}
Therefore we have
$\{z_1,u_1,u_2\}\cup S' \in \mathcal{N}(H,\overline{H})\smallsetminus \mathcal{N}(G,\overline{G})$. Hence
\begin{align}
\label{Eq:C1}
|\mathcal{N}(H,\overline{H})\smallsetminus\mathcal{N}(G,\overline{G})|\geq  2^{\deg_G{u_2}}-1\,,
\end{align}
which implies that
\begin{align}
\label{Eq:C2}
|\mathcal{N}(H,\overline{H})|\geq|\mathcal{N}(G,\overline{G})|\,.
\end{align}
The inequality in~\eqref{Eq:C1} and thus in ~\eqref{Eq:C2} is strict for the following reasons:
\begin{itemize}
\item If $\deg_G(u_1)>2$, then there is $z_2\in V(U_1) \cap N_G(u_1)$ such that $z_2 \neq z_1$. This vertex $z_2$ additionally contributes at least $2^{\deg_G{u_2}}-1$ to the bound in~\eqref{Eq:C1}.
\item If $\deg_G(u_1)=2$ and $u_2$ has a neighbor $y$ that is not adjacent to $c$, then an additional $1$, which accounts for the set $\{c, u_1,u_2,y\}$, can be added to the bound in~\eqref{Eq:C1}.
\item If $\deg_G(u_1)=2$ and all neighbors of $u_2$ are adjacent to $c$ and there is $y'\in V(U_2-u_2)$ such that $y'$ is not adjacent to $c$ in $G$, then $y'$ is not a neighbor of $u_2$ in $G$. Also, $y'$ and $c$ have a common neighbor, otherwise $d_G(y',c)=3$ and thus $\diam (G)=4$: a contradiction to 2). 
\begin{itemize}
\item If $N_G(u_2)\cap N_G(y')\neq \emptyset$, then for every $t\in N_G(u_2)\cap N_G(y')$, we have $\{u_1,u_2,t,y'\} \in \mathcal{N}(H,\overline{H}) \smallsetminus \mathcal{N}(G,\overline{G})$: it induces a $P_4$ in $H$ and thus also in $\overline{H}$, while it induces a disconnected subgraph in $G$ since  $u_1$ an $u_2$ are present but not $c$. In this case, the inequality in~\eqref{Eq:C1} is strict. 
\item If $N_G(u_2)\cap N_G(y')= \emptyset$, then there is a neighbor of $c$ in $U_2$ that is not adjacent to $u_2$ in $G$. Thus for every $t \in (N_{U_2}(c) \smallsetminus N_G(u_2)) \cap N_G(y')$, we have $\{u_1,u_2,c,t\} \in \mathcal{N}(H,\overline{H}) \smallsetminus \mathcal{N}(G,\overline{G})$: it induces a $P_4$ in both $H$ and $\overline{H}$, while $\langle \{u_1,u_2,c,t\}\rangle_{G}$ is disconnected ($u_2$ is isolated in it). In this case, the inequality in~\eqref{Eq:C1} is strict. 
\end{itemize} 
\end{itemize}

Since we obtain a contradiction in all possible cases, we conclude that $\deg_G(u_1)=1$ and thus every cut vertex of $G$ (resp. $\overline{G}$) is adjacent to a pendent vertex of $G$ (resp. $\overline{G}$) as $u_1 \in N_G(c)$. In particular, $G$ has no pendent path of length at least $3$: if $x,y,z,t$ are the last vertices of such path, appearing in this order and with $\deg_G(t)=1$, then $y$ would be a cut vertex that is not adjacent to a pendent vertex. Then by Lemma~\ref{Nopend2}, neither $G$ nor $\overline{G}$ has a pendent path of length at least $2$.

\medskip
Our next step is to prove that neither $G$ nor $\overline{G}$ has a pendent vertex. As was mentioned above, this implies that $G$ and $\overline{G}$ also do not have a cut vertex. This will complete the proof of 3). For this, we prove that if $G$ has a pendent vertex $u$ (and diameter $3$), then one can get a graph $H$ from $G$, by removing $u$ and attaching it to a set of other vertices such that $H$ satisfies
$$
\eta(G)+\eta(\overline{G})<\eta(H)+\eta(\overline{H})\,.
$$
But first, we introduce more notation. Suppose that $X$ and $Y$ are disjoint subsets of $V(G)$, and let $G'$ be obtained from $\langle X\cup Y\rangle_G$ by adding a new vertex $\mu$ adjacent to all (and only) vertices in $X$. We denote by $s_G(X, Y)$ the number of sets $Z\subseteq X\cup Y$ such that $Z\cap X\neq \emptyset,~Z\cap Y\neq \emptyset$ and $\langle \{\mu\}\cup Z\rangle_{G'}$ is connected, by $r_G(X,Y)$ the number of sets $Z\subseteq X\cup Y$ such that $Z\cap X\neq \emptyset,~Z\cap Y\neq\emptyset$ and $\langle Z\rangle_G$ has at least one edge from a vertex in $X$ to a vertex in $Y$. Clearly, $$s_G(X,Y)\leq r_G(X,Y)=r_G(Y,X)\,.$$

Suppose (to the contrary) that $G$ has a pendent vertex, say $u$. Denote by $v$ the neighbor of $u$ in $G$. Let $B=V(G)\smallsetminus(N_G[v])$ and $A=N_G(v)\smallsetminus \{u\}$. Clearly, $A \neq \emptyset$ (resp. $B\neq \emptyset$), otherwise $\langle\{u,v\}\rangle_G$ is isolated in $G$ (resp. $v$ is isolated in $\overline{G}$). Consider the connected graph $$H:=(G-uv)+\{ux: x\in B\}\,.$$ We prove that either $\N(H) +\N(\overline{H}) > \N(G) +\N(\overline{G})$ or $\N(H) +\N(\overline{H}) = \N(G) +\N(\overline{G})$ and $\diam(H)>3$. In each case, we get a contradiction to either the maximality of $G$, or to part 2) of the theorem. Clearly, $G-u$ is isomorphic to $H-u$; so we only compare the numbers of connected induced subgraphs that contain $u$. Note that$N_{\overline{G}}(u)=A\cup B, N_{\overline{G}}(v)=B$ and hence $\N(\overline{G}-v)_u=2^{n-2}$ and $\N(\overline{G})_{u,v}=(2^{|B|}-1)2^{|A|}$. So we have
\begin{align*}
\eta(G)_u=1+ \N(G-u)_v=1+2^{|A|}+s_{G}(A, B)\,,\\
\eta(\overline{G})_u=\N(\overline{G}-v)_u + \N(\overline{G})_{u,v}= 2^{n-2}+(2^{|B|}-1)2^{|A|}\,,
\end{align*}
and %hence 
$$
\eta(G)_u+\eta(\overline{G})_u=1+2^{n-1}+s_G(A,B)\,.
$$
On the other hand, $N_H(u)=B, N_{\overline{H}}(u)=A\cup\{v\}, uv\in E(\overline{H}), N_{\overline{H}}(v)=B\cup\{u\}, \eta(H-v)_u=2^{|B|}+s_G(B,A), \eta(H)_{u,v}=r_G(B,A), \eta(\overline{H}-v)_u=2^{|A|}+s_{\overline{G}}(A,B), \eta(\overline{H})_{u,v}=2^{n-2}$, and therefore
\begin{align*}
\eta(H)_u&=\eta(H-v)_u+\eta(H)_{u,v}=2^{|B|}+s_G(B,A)+r_G(B,A),\\
\eta(\overline{H})_u&=\eta(\overline{H}-v)_u+\eta(\overline{H})_{u,v}=2^{|A|}+s_{\overline{G}}(A,B)+2^{n-2}\,.
\end{align*}
By Remark~\ref{Rem:conn},
\begin{align}\label{Eq:Neq1}
s_G(B,A)+s_{\overline{G}}(A,B)\geq (2^{|B|}-1)(2^{|A|}-1)=2^{n-2}-2^{|A|}-2^{|B|}+1\,,
\end{align}
where $(2^{|B|}-1)(2^{|A|}-1)$ counts the subsets of $A\cup B$ that have at least one element in $A$ and at least one element in $B$.
Therefore,
\begin{align*}%\label{Eq:Neq2}
\eta(H)_u+\eta(\overline{H})_u
&=2^{|B|}+s_G(B,A)+r_G(B,A)+2^{|A|}+s_{\overline{G}}(A,B)+2^{n-2}\\
&\geq 2^{|B|}+r_G(B,A)+2^{|A|}+2^{n-2}-2^{|A|}-2^{|B|}+1+2^{n-2}\\
&\geq 1 +2^{n-1}+r_G(B,A) \geq \eta(G)_u+\eta(\overline{G})_u\,.
\end{align*}
Equality occurs in the last inequality if and only if $r_G(B,A)=s_G(A,B)$, while the bound in~\eqref{Eq:Neq1} is attained if and only if there is no subset of $A\cup B$ that is counted in both $s_G(B,A)$ and $s_{\overline{G}}(A,B)$.

Note that in the graph $G$ (thus in $H$) every element in $B$ has a neighbor in $A$, otherwise $\diam (G)\geq 4$. If $\eta(H)_u+\eta(\overline{H})_u>\eta(G)_u+\eta(\overline{G})_u$, then we obtain a contradiction to the choice of $G$. Otherwise, $\eta(H)_u+\eta(\overline{H})_u=\eta(G)_u+\eta(\overline{G})_u$ in which case we have $r_G(B,A)=s_G(A,B)$ and there is no subset of $A\cup B$ that is counted in both $s_G(B,A)$ and $s_{\overline{G}}(A,B)$; in particular, $A\cup B$ is counted in only one of $s_G(B,A)$ and $s_{\overline{G}}(A,B)$.  

Suppose that $A\cup B$ is not counted in $s_G(B,A)$. Then  there is a nonempty set $A_1\subseteq A$ such that $xz\notin E(H)$ (thus $xz\notin E(G)$) for all $x\in A_1$ and $z\in B\cup (A\smallsetminus A_1)$. So $v$ is a cut vertex of $H$. The set $A_2=A\smallsetminus A_1$ is nonempty, otherwise $G$ is disconnected. For every $x\in A_1$, we have $d_H(u,x)=1+d_H(b,a_2)+d_H(a_2,x)$ for some $b\in B$ and some $a_2 \in A_2$. Thus $d_H(u,x)\geq 4$, contradicting part 2). Therefore, $A\cup B$ is indeed counted in $s_G(B,A)$.

Suppose that $A\cup B$ is not counted in $s_{\overline{G}}(A,B)$. Then there is a nonempty set $B_1\subseteq B$ such that $yz \notin E(\overline{H})$ (thus $yz\in E(H)$ and $yz\in E(G)$) for all $y\in B_1$ and $z\in A\cup (B\smallsetminus B_1)$. So  $v$ is  a  cut vertex of $\overline{H}$: in the graph $\overline{H}-v$ there is no path to join a vertex in $B_1$ to a vertex in $V(\overline{H})\smallsetminus B_1$. Thus, $v$ must be adjacent to some pendent vertex $w$ in $\overline{H}$. We choose $B_1$ to be as large (in terms of cardinality) as possible, so that $w$ has to be in $B_1$. This means that $N_H(w)=V(H) \smallsetminus \{v,w\}$ and that $N_G(w)=A\cup (B\smallsetminus \{w\})$. The set $B\smallsetminus \{w\}$ is nonempty, otherwise $\langle w,v,u \rangle_{\overline{H}}$ is a pendent path of length $2$: impossible since $H$ is a maximal graph, i.e. $\eta(H)+\eta(\overline{H}) \geq \eta(J)+\eta(\overline{J})$ for all connected graphs $J$ of order $|V(H)|$.

Define $H_1$ to be the connected graph
$$H_1=G-(\{uv\}\cup \{wx: x\in A\})+ (\{ux:x\in A\} \cup \{uw\}) \,.$$ Since $G-u-w$ and $H_1-u-w$ are isomorphic graphs, we have to compare only the numbers of connected induced subgraphs that contain $u$ or $w$. Set $|A|=a$ and $|B|-1=b$. 
\begin{figure}[htbp]
\centering
\definecolor{qqqqff}{rgb}{0.,0.,1.}
\begin{tikzpicture}
\node[fill=black,label=above:{$u$},circle,inner sep=1pt] (t1) at (1.25,0) {};
\node[fill=black,label=above:{\small $v$},circle,inner sep=1pt] (t2) at (2.2,-0.5) {};
\node[fill=black,label=above:{\small $w$},circle,inner sep=1pt] (t5) at (6.75,1.75) {};
\draw[dashed] (3.5,-1.5) ellipse (0.75cm and 0.75cm);
\draw (t1)--(t2)--(3.7,-0.76);
\draw (t2)--(2.75,-1.5);
\node at (3.75,-1.5) {$A$};
\draw[dashed] (6.5,-1.5) ellipse (0.8cm and 0.8cm);
\node at (6.5,-1.45) {$B\smallsetminus \{w\}$};
\draw(3.73,-0.76)--(t5);
\draw(4.1,-2)--(t5);
\draw(t5)--(5.75,-1.5);
\draw(t5)--(7.25,-1.5);
\draw[dashed](3.5,-0.75)--(6.5,-0.75);
\draw[dashed](3.7,-2.25)--(6.4,-2.25);
%%%%%%%%2.2,-0.5
\node[fill=black,label=above:{$u$},circle,inner sep=1pt] (t1) at (1.25+8,0) {};
\node[fill=black,label=above:{\small $v$},circle,inner sep=1pt] (t2) at (2.2+8,-0.55) {};
\node[fill=black,label=above:{\small $w$},circle,inner sep=1pt] (t5) at (6.75+8,1.75) {};
\node at (3.75+8,-1.5) {$A$};
\draw[dashed] (6.5+8,-1.5) ellipse (0.8cm and 0.8cm);
\draw[dashed] (3.5+8,-1.5) ellipse (0.75cm and 0.75cm);
\node at (6.5+8,-1.45) {$B\smallsetminus \{w\}$};
\draw[dashed](3.5+8,-0.75)--(7+8,-0.7);
\draw[dashed](3.7+8,-2.25)--(6.4+8,-2.25);
\draw (9.25,0) .. controls +(60:3cm) and +(150:1cm) .. (t5);
\draw (9.25,0) .. controls +(50:5cm) and +(90:1cm) .. (13.8,-1.5);
\draw (9.25,0) .. controls +(50:5cm) and +(90:2cm) .. (15.25,-1.5);
\draw (t2) .. controls +(-90:4.2cm) and +(-60:7cm) .. (t5);
\draw (t2) .. controls +(-90:3cm) and +(-90:1cm) .. (13.7,-1.5);
\draw (t2) .. controls +(-90:3.5cm) and +(-90:2cm) .. (15.3,-1.5);
\draw (t1) .. controls +(50:2cm) and +(90:0.5cm) .. (10.75,-1.5);
\draw (t1) .. controls +(50:2cm) and +(90:1cm) .. (12.25,-1.5);
\node at (4,-3.5) {$G$};
\node at (4+8,-3.5) {$\overline{G}$};
%%%%%%%%%%%%%%%%%%%%%%%%%%%%%%%%%%%
\node[fill=black,label=above:{$u$},circle,inner sep=1pt] (t1) at (1.25,0-6) {};
\node[fill=black,label=above:{\small $v$},circle,inner sep=1pt] (t2) at (2.2,-0.5-6) {};
\node[fill=black,label=above:{\small $w$},circle,inner sep=1pt] (t5) at (6.75,1.75-6) {};
\draw[dashed] (3.5,-1.5-6) ellipse (0.75cm and 0.75cm);
\draw (t2)--(3.7,-0.76-6);
\draw (t2)--(2.75,-1.5-6);
\node at (3.75,-1.5-6) {$A$};
\draw[dashed] (6.5,-1.5-6) ellipse (0.8cm and 0.8cm);
\node at (6.5,-1.45-6) {$B\smallsetminus \{w\}$};
\draw(t5)--(5.75,-1.5-6);
\draw(t5)--(7.25,-1.5-6);
\draw[dashed](3.5,-0.75-6)--(6.5,-0.75-6);
\draw[dashed](3.7,-2.25-6)--(6.4,-2.25-6);
\draw (t1) .. controls +(50:1cm) and +(90:1cm) .. (2.75,-1.5-6);
\draw (t1) .. controls +(50:1cm) and +(90:2cm) .. (4.25,-1.5-6);
\draw (t1) .. controls +(60:3cm) and +(150:1cm) .. (t5);
%%%%%%%%2.2,-0.5
\node[fill=black,label=above:{$u$},circle,inner sep=1pt] (t1) at (1.25+8,0-6) {};
\node[fill=black,label=above:{\small $v$},circle,inner sep=1pt] (t2) at (2.2+8,-0.55-6) {};
\node[fill=black,label=above:{\small $w$},circle,inner sep=1pt] (t5) at (6.75+8,1.75-6) {};
\node at (3.75+8,-1.5-6) {$A$};
\draw[dashed] (6.5+8,-1.5-6) ellipse (0.8cm and 0.8cm);
\draw[dashed] (3.5+8,-1.5-6) ellipse (0.75cm and 0.75cm);
\node at (6.5+8,-1.45-6) {$B\smallsetminus \{w\}$};
\draw[dashed](3.5+8,-0.75-6)--(7+8,-0.7-6);
\draw[dashed](3.7+8,-2.25-6)--(6.4+8,-2.25-6);
\draw (9.25,0-6) .. controls +(50:5cm) and +(90:1cm) .. (13.8,-1.5-6);
\draw (9.25,0-6) .. controls +(50:5cm) and +(90:2cm) .. (15.25,-1.5-6);
\draw (t2) .. controls +(-90:4.2cm) and +(-60:7cm) .. (t5);
\draw (t2) .. controls +(-90:3cm) and +(-90:1cm) .. (13.7,-1.5-6);
\draw (t2) .. controls +(-90:3.5cm) and +(-90:2cm) .. (15.3,-1.5-6);
\draw (t5)--(11,-0.9-6);
\draw (t5)--(12.1,-2-6);
\node at (4,-3.5-6) {$H_1$};
\node at (4+8,-3.5-6) {$\overline{H_1}$};
\end{tikzpicture}
\caption{Sketches of the graphs $G$ and $\overline{G}$, enough to get Equation~\eqref{Eq:Gw}; Sketches of the graphs $H_1$ and $\overline{H_1}$, enough to get Equation~\eqref{Eq:H1}. Dashed connection means not necessarily all possible edges are present.}\label{Fig:ExtraExp}
\end{figure}
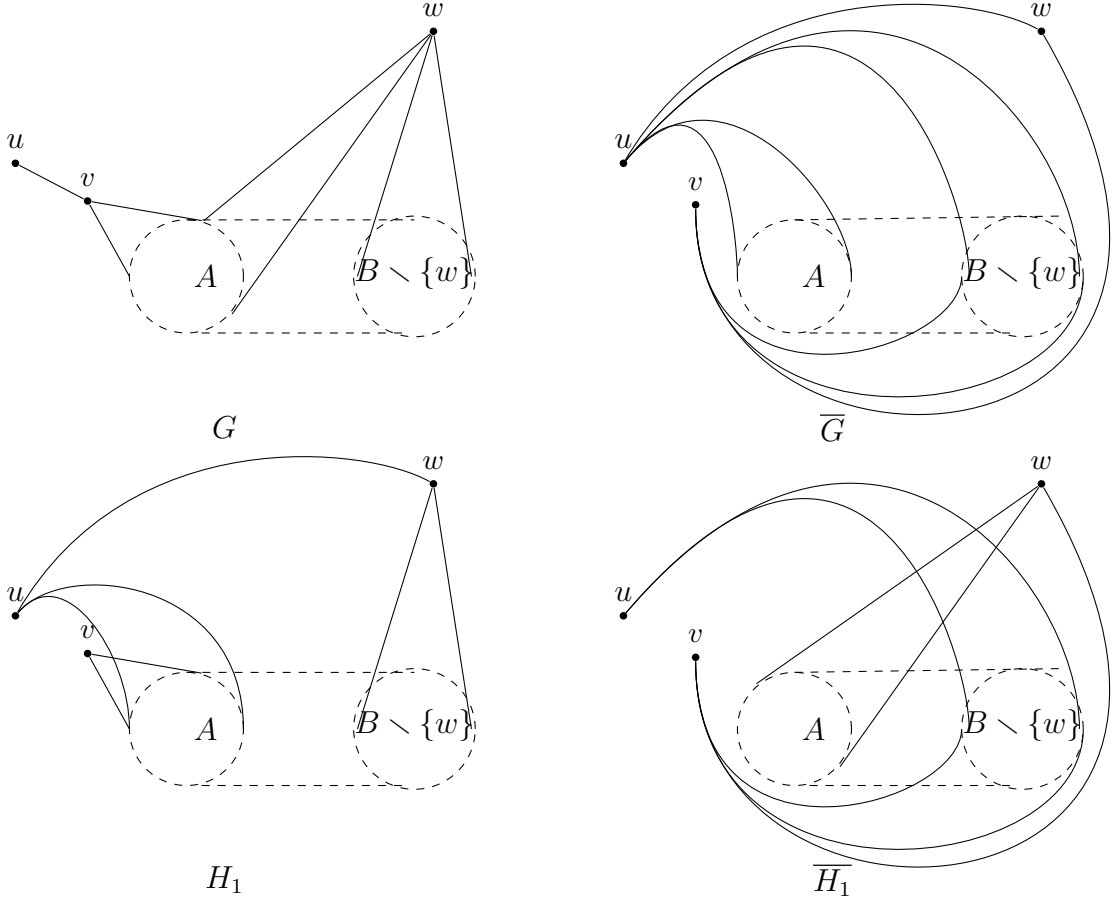
We have (see Figure~\ref{Fig:ExtraExp} for sketches of $G$ and $\overline{G}$)
\begin{align}
 \label{Eq:Gw}
\eta(G)_w&=\eta(G-v)_w + \eta(G)_{v,w} = 2^{a+b}+2(2^a-1)2^b\,, \\
\eta(G-w)_u&=1+ \eta(G-w)_{u,v}= 1+2^a + s_G(A,B\smallsetminus\{w\})\,,\nonumber\\
\eta(\overline{G})_w&=\eta(\overline{G})_{u,w} + \eta(\overline{G}-u)_w= 2\cdot 2^{a+b} + 1+ 2^b+s_{\overline{G}}(B\smallsetminus\{w\},A)\,, \nonumber\\
\eta(\overline{G}-w)_u&=  \eta(\overline{G}-w-v)_u + \eta(\overline{G}-w)_{u,v}  =2^{a+b}+2^a(2^b-1)\,.\nonumber
\end{align}
Thus
\begin{align*}
\eta(G)_w+\eta(G-w)_u+ & \eta(\overline{G})_w + \eta(\overline{G}-w)_u\\
&=7\cdot 2^{a+b}+ 2-2^b+ s_G(A,B\smallsetminus\{w\})+s_{\overline{G}}(B\smallsetminus\{w\},A)\,.
\end{align*}
On the other hand (see Figure \ref{Fig:ExtraExp} for sketches of $G$ and $\overline{G}$), 
\begin{align}
\label{Eq:H1}
\eta(H_1)_w&=\eta(H_1-u-v)_w + \eta(H_1-v)_{u,w} + \eta(H_1-u)_{v,w} + \eta(H_1)_{u,v,w}\\
& =2^b+s_{G}(B\smallsetminus\{w\},A)+ 2^{a+b} + r_G(B\smallsetminus\{w\},A)+(2^a-1)2^b\,,\nonumber\\
\eta(H_1-w)_u&=2^a +2^a-1 + 2\cdot s_G(A,B\smallsetminus\{w\})\nonumber\,,
\end{align}
and 
\begin{align*}
\eta(\overline{H_1})_w &=\eta(\overline{H_1}-u-v)_w + \eta(\overline{H_1}-v)_{u,w} + \eta(\overline{H_1})_{v,w} \\
&=2^a+s_{\overline{G}}(A,B\smallsetminus\{w\})+r_{\overline{G}}(A,B\smallsetminus\{w\})+2\cdot 2^{a+b} \,,\\
\eta(\overline{H_1}-w)_u&=2\cdot 2^b+2\cdot s_{\overline{G}}(B\smallsetminus\{w\},A)\,.
\end{align*}
Thus
\begin{align*}
\eta(H_1)_w+  & \eta(H_1-w)_u+  \eta(\overline{H_1})_w + \eta(\overline{H_1}-w)_u\\
&=4\cdot 2^{a+b}+ 2\cdot s_G(A,B\smallsetminus\{w\})+2\cdot s_{\overline{G}}(B\smallsetminus\{w\},A) + r_G(B\smallsetminus\{w\},A) \\
& + r_{\overline{G}}(A,B\smallsetminus\{w\}) + s_G(B\smallsetminus\{w\},A) + s_{\overline{G}}(A,B\smallsetminus\{w\}) +2^{b+1} + 3\cdot 2^a -1\,,
\end{align*}
and the inequality $s_G(X,Y)+s_{\overline{G}}(Y,X)\geq (2^{|X|}-1)(2^{|Y|}-1)$ yields
\begin{align*}
\eta(H_1)_w+  & \eta(H_1-w)_u+  \eta(\overline{H_1})_w + \eta(\overline{H_1}-w)_u \\
& \geq  7\cdot 2^{a+b} +2 -2^b +r_{\overline{G}}(B\smallsetminus\{w\},A)+r_G(A, B\smallsetminus\{w\})\\
& \geq  7\cdot 2^{a+b} +2 -2^b +s_{\overline{G}}(B\smallsetminus\{w\},A)+s_G(A,B\smallsetminus\{w\})\\
&=\eta(G)_w+\eta(G-w)_u+  \eta(\overline{G})_w + \eta(\overline{G}-w)_u\,.
\end{align*}

If 
\begin{align*}
\eta(H_1)_w+  \eta(H_1-w)_u+  & \eta(\overline{H_1})_w + \eta(\overline{H_1}-w)_u \\
& > \eta(G)_w+\eta(G-w)_u+  \eta(\overline{G})_w + \eta(\overline{G}-w)_u\,,
\end{align*}
then we are done, i.e. the proof of 3) is complete. Otherwise, 
\begin{align*}
\eta(H_1)_w+  \eta(H_1-w)_u+  & \eta(\overline{H_1})_w + \eta(\overline{H_1}-w)_u \\
& = \eta(G)_w+\eta(G-w)_u+  \eta(\overline{G})_w + \eta(\overline{G}-w)_u\,,
\end{align*}
which implies, in particular, that $s_G(A,B\smallsetminus\{w\})=r_G(A, B\smallsetminus\{w\})$, $s_{\overline{G}}(B\smallsetminus\{w\},A)=r_{\overline{G}}(B\smallsetminus\{w\},A)$, $A \cup B\smallsetminus\{w\}$ is counted in only one of $s_G(A,B\smallsetminus\{w\})$ and $s_{\overline{G}}(B\smallsetminus\{w\},A)$, and  $A \cup B\smallsetminus\{w\}$ is counted in only one of $s_G(B\smallsetminus\{w\},A)$ and $s_{\overline{G}}(A,B\smallsetminus\{w\})$.

If $A \cup B\smallsetminus\{w\}$ is not counted in $s_G(A,B\smallsetminus\{w\})=r_G(A, B\smallsetminus\{w\})$, then in the graph $G$ there is no edge between a vertex in $A$ and a vertex in $B\smallsetminus\{w\}$. In this case $d_G(u,\nu)=4$ for any $\nu \in B\smallsetminus \{w\}$ (a contradiction).

If $A \cup B \smallsetminus\{w\}$ is not counted in $s_{\overline{G}}(B\smallsetminus\{w\},A)=r_{\overline{G}}(B\smallsetminus\{w\},A)$, then there is no edge of $\overline{G}$ between a vertex in $A$ and a vertex in $B\smallsetminus\{w\}$. Fix $x\in A$ such that $x$ has the minimum degree in $G$ among all vertices in $A$. Define a new connected graph $H_2:=G-wx$. $H_2$ is connected because $N_{H_2}(w)=V(H_2-u-v-x)$ and $uv,vx,vy$ are edges of $H_2$ for any $y\in A\smallsetminus \{x\}$. The vertex $y$ exists, otherwise $A=\{x\}$ and $G$ has a pendent path of length $2$. To compare $\N(H_2)+\N(\overline{H_2})$ with $\N(G)+\N(\overline{G})=\N(H_1)+\N(\overline{H_1})$, it suffices to compare the number of connected induced subgraphs of $G$ that have $wx$ as a cut edge with those of $\overline{H_2}$ that have $wx$ as a cut edge. We do so in the following. Bearing in mind that $N_G(w)=A\cup (B\smallsetminus \{w\})$ and that $B \smallsetminus\{w\} \subset N_G(x)$, it is easy to see that the number of connected induced subgraphs of $G$ that have $wx$ as a cut edge is precisely $2+ 2^{|A \smallsetminus N_G[x]|}$. On the other hand, since $N_{\overline{H_2}}[w]=\{u,v,w\}$ and $N_{\overline{H_2}}(x)\smallsetminus \{u,w\} \subseteq A$, the number of connected induced subgraphs of $\overline{H_2}$ that have $wx$ as a cut edge is at least $$ 2 \cdot 2^{|A \smallsetminus N_G[x]|} + (2^b -1)2^{|A \smallsetminus N_G[x]|}\,.$$
This expression is obtained as follows. The number of choices that do not involve an element of $\{u\}\cup B\smallsetminus \{w\}$ is given by $2\cdot 2^{|A \smallsetminus N_G[x]|}$, the vertex $v$ can be included or not. The other term counts those connected induced subgraphs that involve both $v$ and at least an element of $B\smallsetminus \{w\}$.

%\textcolor{blue}{ $2 \cdot 2^{|A \smallsetminus N_G[x]|}$ is made of the $2$ choices to include $v$ or not and $2^{|A \smallsetminus N_G[x]|}$ choices to include some neighbors of $x$ in $\overline{H_2}$ or not. $(2^b -1)2^{|A \smallsetminus N_G[x]|}$ counts those that contain at least one element of $B\smallsetminus\{w\}$ while neighbors of $x$ in $\overline{H_2}$ can be included or not; they necessarily contain $v$ to be connected}.
 
Furthermore, since $b\geq 1$,
\begin{align*}
 2 \cdot 2^{|A \smallsetminus N_G[x]|} + (2^b -1)2^{|A \smallsetminus N_G[x]|} \geq 2+ 2^{|A \smallsetminus N_G[x]|}
\end{align*}
with equality if and only if $A \smallsetminus N_G[x] =\emptyset$ and $b=1$. 

Finally, suppose that  $A \smallsetminus N_G[x] =\emptyset$ and $b=1$. Then the set $A \cup B$ induces a complete graph in $G$, given that $x$ was chosen to have the minimum degree in $G$ among all vertices in $A$. We fix $y\in A$ and define a new graph $H_3:=G-vy$. The graph $H_3$ is connected since $N_G(v)=A\cup \{u\}$ and $|A|>2$. It is easy to see that the number of connected induced subgraphs of $G$ that have $vy$ as a cut edge is $2^3=8$, and that those of $\overline{H_3}$ that have $vy$ as a cut edge is $2^2 + 2^{a-1}$. Furthermore, $4 + 2^{a-1}>8$ since $n= a+4 \geq 8$. This shows that $\N(H_3)+\N(\overline{H_3})>\N(G)+\N(\overline{G})$, a contradiction to the maximality of $G$.

The theorem follows.

 \end{proof}

\section{$n$-vertex trees}\label{Sec:Tree}
If we restrict the study to trees, the lower bound $2^n+n-1$ in~\eqref{Eq:NOR} (Theorem~\ref{Th:GandGbar}) cannot be improved. This is because stars do not have an induced $P_4$ and thus they attain that lower bound. In fact, $S_n$ is the unique tree that minimises $\N(T)+\N(\overline{T})$ among all $n$-vertex trees $T$; see Corollary~\ref{Coro:Treeslower}. 

\medskip
In this section, we provide a sharp upper bound for $\N(T)+\N(\overline{T})$ where $T$ is a tree of order $n$. We begin with an intermediate result (Lemma~\ref{Lem:Stretch} below), which suggests that a maximal tree (a tree $T$ of order $n$ that maximises $\N(T)+\N(\overline{T})$) must have diameter $3$.

\begin{lemma}\label{Lem:Stretch}
Let $G$ be a connected graph, %of order at least $2$, 
and $v\in V(G)$. Fix an integer $k\geq 1$. Let $H$ be the graph that results from merging $v$ with a pendent vertex $u$ of $S_{k+2}$, and $K$ the graph that results from merging $v$ with the center $c$ of $S_{k+2}$. Denote by $r$ the degree of $v$ in $\overline{G}$. If $r\geq 2$, then $$\N(K)+\N(\overline{K}) - (\N(H)+\N(\overline{H}))>0\,.$$.
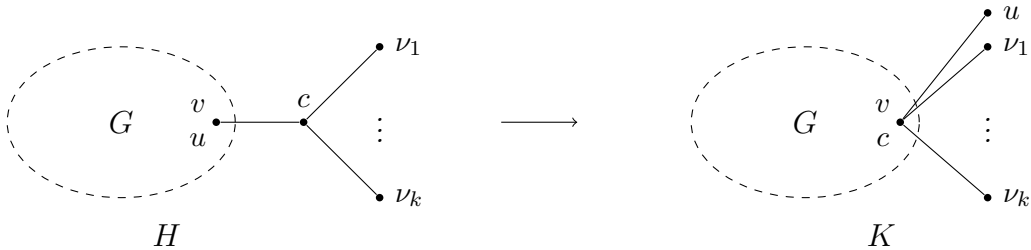
\begin{figure}[htbp]
\centering
\definecolor{qqqqff}{rgb}{0.,0.,1.}
\begin{tikzpicture}
\draw[dashed] (0,0) ellipse (1.5cm and 1cm);
\node[fill=black,label=left:{\small $\begin{array}{c}v\\u\end{array}\hspace{-0.25cm}$},circle,inner sep=1pt] (t1) at (1.25,0) {};
\node[fill=black,label=above:{\small $c$},circle,inner sep=1pt] (t2) at (2.4,0) {};
\node[fill=black,label=right:{\small $\nu_1$},circle,inner sep=1pt] (t3) at (3.4,1) {};
\node[fill=black,label=right:{\small $\nu_k$},circle,inner sep=1pt] (t4) at (3.4,-1) {};
\node at (3.4,0) {$\vdots$};
\node at (0,0) {$G$};
\draw (t1)--(t2)--(t3);
\draw (t2)--(t4);
\node at (0.6,-1.5) {$H$};
%%%%%%%%%%%%%
 \draw[->] (5,0)--(6,0);
%%%%%%%%%%%%%
\draw[dashed] (0+9,0) ellipse (1.5cm and 1cm);
\node[fill=black,label=left:{\small $\begin{array}{c}v\\c\end{array}\hspace{-0.25cm}$},circle,inner sep=1pt] (t1) at (1.25+9,0) {};
\node[fill=black,label=right:{\small $\nu_1$},circle,inner sep=1pt] (t3) at (3.4+8,1) {};
\node[fill=black,label=right:{\small $\nu_k$},circle,inner sep=1pt] (t4) at (3.4+8,-1) {};
\node[fill=black,label=right:{\small $u$},circle,inner sep=1pt] (t5) at (3.4+8,1.45) {};
\node at (3.4+8,0) {$\vdots$};
\node at (0+9,0) {$G$};
\draw (t3)--(t1)--(t4);
\draw (t1)--(t5);
\node at (10,-1.5) {$K$};
\end{tikzpicture}
\caption{The graphs $H$ and $K$ described in Lemma~\ref{Lem:Stretch}.}
\label{Fig:Stretch}
\end{figure}
\end{lemma}

\begin{proof}
Set $n=|V(H)|=|V(K)|$. The number of connected induced subgraphs of $H$ that do not contain the edge $uc$ is $2^k+k+\N(G)$ and thus $\N(H)=2^k+k+\N(G)+2^k \N(G)_v$. The number of connected induced subgraphs of $K$ that contain vertex $v$ is $2^{k+1}\N(G)_v$ and thus $\N(K)=2^{k+1}\N(G)_v+ k+1+\N(G-v)$.  With $\eta(G)-\eta(G-v)=\eta(G)_v$, these give
\begin{align}\label{Aster1}
\N(K)-\N(H)=(2^k-1)(\N(G)_v-1)\,.
\end{align}
The complement $\overline{H}$ of $H$ can be obtained by first adding all possible edges between a vertex of $K_k$ and a vertex of $\overline{G}$, and then adding all edges from an isolated vertex $c$ to a vertex of $\overline{G}-v$. By counting the connected induced subgraphs of $\overline{H}$ that contain $c$ but not $v$, $v$ but not $c$, both $c$ and $v$, and the rest, we obtain 
\begin{align*}
\eta(\overline{H}-v)_c&=1+2^k(2^{n-k-2}-1),\\
\eta(\overline{H}-c)_v&=\N(\overline{G})_v+2^{n-k-2}(2^{k}-1),\\
\eta(\overline{H})_{v,c}&=2^{n-k-2-r}(2^r-1)+(2^{n-k-2}-1)(2^{k}-1)\,,
\end{align*}
and $\N(\overline{H}-v-c)$, respectively. The quantity $2^{n-k-2-r}(2^r-1)+(2^{n-k-2}-1)(2^{k}-1)$ is obtained by first counting the connected induced subgraphs that contain none of the vertices of $K_k$: there are $2^{n-k-2-r}(2^r-1)$ of them since $\deg_{\overline{G}}v=r$. Such a subgraph has to contain a neighbor of $v$ in $G-v$ to be connected.
	
The complement $\overline{K}$ of $K$ can be obtained by taking the disjoint union of $K_{k+1}$ and $\overline{G}$, and adding all possible edges between $K_{k+1}$ and $\overline{G}-v$. Let $u$ be a fixed vertex of $K_{k+1}$ in $\overline{K}$. By counting the connected induced subgraphs of $\overline{K}$ that contain $u$ but not $v$, $v$ but not $u$, both $u$ and $v$, and the rest, we obtain $2^{n-2}$,
\begin{align*}
&\N(\overline{G})_v+2^{n-k-2-r}(2^{r}-1)(2^{k}-1)\,,\\
&2^{n-2-r}(2^r-1)\,,
\end{align*}
and $\N(\overline{K}-v-u)$, respectively. Note that $(\overline{H}-v)-c$ and $(\overline{K}-v)-u$ are isomorphic graphs. Thus,
\begin{align}\label{Aster2}
\begin{split}
\N(\overline{K})- \N(\overline{H})&=2^{n-2}+2^{n-k-2-r}(2^{r}-1)(2^{k}-1)+2^{n-2-r}(2^r-1)\\
& \hspace*{-2.5cm} - \big(1+2^k(2^{n-k-2}-1)+2^{n-k-2}(2^{k}-1)+ 2^{n-k-2-r}(2^r-1)+(2^{n-k-2}-1)(2^{k}-1)\big)\\
&=2(2^k -1) + 2^{-r}(2^{n-k-1}-2^{n-1})
\end{split}
\end{align}
after simplification. It follows from~\eqref{Aster1} and~\eqref{Aster2} that
$$
\N(K)+\N(\overline{K})-(\N(H)+\N(\overline{H}))=(2^k-1)(\N(G)_v+1)+2^{-r}(2^{n-k-1}-2^{n-1}).
$$
Note that the degree of $v$ in $G$ is $n-k-2-r$, and that $n-k-2-r>0$ since $G$ is connected. As per the statement of the lemma, we are only interested in the case where $r\geq 2$. Fix a neighbor $w$ of $v$ in $G$, that has a neighbor $w'\notin N_G(v)$. Such a vertex $w$ must exist, otherwise there would be no path to join $v$ to the other vertices not in $N_G(v)$. The number of connected induced subgraphs of
\begin{itemize}
\item $G-N_G[v]$ that contain $w'$ is at least $r$,
\item  $G-(G-N_G[v])$ that contain both $v$ and $w$ is precisely $2^{n-k-3-r}$.
\end{itemize}
Thus, these subgraphs contribute to $\N(G)_v$ by at least $r\cdot 2^{n-k-3-r}$, obtained by adding the edge $ww'$. On the other hand, the number of connected induced subgraphs that consist of $v$ and a subset of its neighbors in $G$ is precisely $2^{n-k-2-r}$. Therefore,
$$\N(G)_v \geq 2^{n-k-2-r}+r\cdot 2^{n-k-3-r}\,.$$ Hence, with $r\geq 2$, we have $r2^{n-k-3-r}\geq 2^{n-k-2-r}$ and 
\begin{align*}
\N(K)+\N(\overline{K})&-(\N(H)+\N(\overline{H}))\\
& \geq (2^k-1)(2^{n-k-2-r}+r2^{n-k-3-r} +1)+2^{-r}(2^{n-k-1}-2^{n-1})\\
& \geq (2^k-1)(2^{n-k-1-r} +1)+2^{-r}(2^{n-k-1}-2^{n-1})\\
& = 2^k-1>0\,.
\end{align*}
The lemma follows.
\end{proof}

\medskip
For two integers $t$ and $s$ such that $1\leq t \leq s$, we write $H^{t-1}_{s-1}$ for the tree with degree sequence $(s,t,1,1,\ldots,1)$. This tree has $t+s$ vertices of which $t+s-2$ are pendent provided that $t>1$.

\medskip
It is not difficult to see that $P_5$ and $H^1_2$ are both maximal trees for $n=5$. They are in fact the only trees of order $5$ that contain $P_4$. The maximal tree $T$ given in the next theorem satisfies properties 1) and 2) specified in Theorem \ref{GeneGraphs}, namely $T$ and $\overline{T}$ are both connected and of diameter $3$.

\begin{theorem}\label{TREE}
Let $T$ be a tree of order $n>5$. Then 
$$
\N(T)+\N(\overline{T})\leq \N\left(H_{\lceil (n-2)/2\rceil}^{\lfloor (n-2)/2\rfloor} \right)+\N\left(\overline{H_{\lceil (n-2)/2\rceil}^{\lfloor (n-2)/2\rfloor}}\right)=5\cdot 2^{n-2}-2^{\lceil (n-2)/2\rceil}-2^{\lfloor (n-2)/2\rfloor }+n\,.
$$
Equality holds if and only if $T$ and $H_{\lceil (n-2)/2\rceil}^{\lfloor (n-2)/2\rfloor} $ are isomorphic trees.
\end{theorem}

\begin{proof}
First, note that the graph transformation that takes $H$ to $K$ in Lemma~\ref{Lem:Stretch} (see Figure \ref{Fig:Stretch}) increases the number of pendent vertices at least by $1$ and does not increase the diameter. 

In view of Corollary \ref{Coro:Treeslower}, we know that the only tree $T$ of order $n>5$ and diameter $2$ is the star and it satisfies the inequality. For the rest of the proof we only consider trees of diameter at least $3$.

Let $T$ be a tree of order $n>5$ that is not a star. Starting with $T$, we iterate Lemma~\ref{Lem:Stretch} for the newly constructed trees until it is no longer possible to apply Lemma~\ref{Lem:Stretch}. Let $T'$ be the resulting tree.

Suppose that $T'$ has diameter at least $4$. Let  $P$ be a longest (and hence induced) path in $T'$ with vertices $v_1,v_2,v_3,v_4,v_5,\dots,v_k$ from one end vertex to the other. The vertex $v_1$ has to be of degree $1$, and $v_2$ has only one neighbor in $T'$ whose degree is greater than $1$. Hence, we can have a Lemma~\ref{Lem:Stretch} decomposition of $T'$ with $v$ being $v_3$ and thus $r\geq 2$ (in $\overline{T'}$ the vertex $v_3$ is adjacent to $v_5$ and at least another vertex in $V(\overline{T'})\smallsetminus\{v_1,v_2,v_3,v_4,v_5\}$). This contradicts the assumption on $T'$. Hence $T'$ has diameter $3$ and hence $T'$ is isomorphic to $H^{t}_s$ for some $t,s$ such that $t+s=n-2$. Let $u$ and $v$ be the vertices of $H^{t}_s$ such that $\deg_{H^{t}_s}(u)=s$ and $\deg_{H^{t}_s}(v)=t$.

\medskip
The complement of $H^t_s$ can be obtained by joining two isolated vertices $u$ and $v$ to the complete graph $K_{t+s}$ such that $u$ and $v$ have degrees $t$ and $s$, respectively, and their open neighborhoods form a partition of $V(K_{t+s})$. Thus, by categorising subgraphs by the following cases: containing $u$ and $v$; $u$ and not $v$; $v$ and not $u$; neither $u$ nor $v$, we get 
$\N(H_s^t)_{u,v}=2^{t+s}, \N(H_s^t-v)_{u}=2^{s}, \N(H_s^t-u)_{v}=2^{t}, \N(H_s^t-u-v)=t+s$
and thus
$$ \N(H_s^t)=2^{t+s}+2^t+2^s+(t+s)\,.$$ By similar way, we have
$
\N(\overline{H_s^t})_{u,v}=(2^t-1)(2^s-1), \N(\overline{H_s^t}-u)_{v}=1+ 2^t(2^s-1), \N(\overline{H_s^t}-v)_{u}=1+ 2^s(2^t-1), \N(\overline{H_s^t}-u-v)=2^{t+s}-1
$
and 
$$\N(\overline{H_s^t})=(2^t-1)(2^s-1) +  (1+ 2^t(2^s-1)) + (1+ 2^s(2^t-1) )   + (2^{t+s}-1)\,.$$
Therefore, we have
\begin{align}\label{Eq:Hrs}
\N(H_s^t)+\N(\overline{H_s^t})= 2^{n}+ 2^{n-2}-2^t-2^s+n
\end{align}
which attains its maximum $2^{n}+ 2^{n-2} - 2^{\lfloor (n-2)/2\rfloor} -  2^{\lceil (n-2)/2\rceil}+n$ only at $t=\lfloor (n-2)/2\rfloor$. This completes the proof of the theorem.
\end{proof}

\section{Concluding remarks}

In this work, we studied the total number, $\N(G) +\N(\overline{G})$, of connected induced subgraphs of $G$ and its complement $\overline{G}$. In the case where $G$ is an $n$-vertex tree, we determined both the unique $n$-vertex tree that minimises  $\N(G) +\N(\overline{G})$ and the unique $n$-vertex tree that maximises $\N(G) +\N(\overline{G})$, and also gave formulas for the corresponding extreme values. For general graphs $G$ of order $n$, we proved that $\N(G) +\N(\overline{G})$ is minimum if and only if neither $G$ nor $\overline{G}$ has an induced path of length $3$. Unfortunately, the maximisation counterpart appears to be hard, and we only managed to obtain some partial characterisations, namely that $G$ and $\overline{G}$ are both of diameter at most $3$ and none of them has a cut vertex.

\medskip
On the other hand, we conducted an exhaustive computer search among all graphs of order $n$ for small values $n=5,6,7$, see Table \ref{Table1}. This suggests (also based on the ideas of the proof of Theorem~\ref{GeneGraphs}, where often we add an edge between vertices ``too'' distant apart to increase $\eta(G)+\eta(\overline{G})$) the following conjecture:

\begin{conjecture}
Let $n\geq 5$ be an integer. If $G$ is an $n$-vertex graph that satisfies 
$$
\eta(G)+\eta(\overline{G})\geq \eta(H)+\eta(\overline{H})
$$ 
for all graphs $H$ of order $n$, then $G$ is of diameter $2$.
\end{conjecture} 

One reason why proving this conjecture appears difficult is that adding an edge between two vertices at distance $3$ does not always increase the sum of the numbers of connected induced subgraphs of a graph and its complement. For example, by adding an edge that joins the ends of $P_4$ we obtain $C_4$. Since $C_4$ has no induced $P_4$, by Theorem \ref{Th:GandGbar} we have 
$$
\eta(P_4)+\eta(\overline{P_4})>\eta(C_4)+\eta(\overline{C_4})\,.
$$

%Naturally, an answer to this open conjecture (or probably to characterising a maximal graph) would be very welcome.

\end{document}